\documentclass[12pt]{amsart}  

\usepackage[latin1]{inputenc}
\usepackage{amsmath} 
\usepackage{amsfonts}
\usepackage{amssymb}
\usepackage{stmaryrd}
\usepackage{latexsym} 
\usepackage{graphicx}
\usepackage{subfigure}
\usepackage{color}
\usepackage{hyperref}
\usepackage{verbatim}
\usepackage[all]{xy}
\usepackage{graphics}
\usepackage{pdfsync}
\usepackage{bm}

\theoremstyle{definition}
\newtheorem{theorem}{Theorem}[section]
\newtheorem{proposition}[theorem]{Proposition}

\newtheorem{lemma}[theorem]{Lemma}
\newtheorem{corollary}[theorem]{Corollary}
\newtheorem{conjecture}[theorem]{Conjecture}
\newtheorem{definition}[theorem]{Definition}
\newtheorem{example}[theorem]{Example}

\theoremstyle{remark}

\numberwithin{equation}{section}
\newlength\cellsize 
\savebox2{%
\begin{picture}(15,15)
\put(0,0){\line(1,0){15}}
\put(0,0){\line(0,1){15}}
\put(15,0){\line(0,1){15}}
\put(0,15){\line(1,0){15}}
\end{picture}}
\newcommand\cellify[1]{\def\thearg{#1}\def\nothing{}%
\ifx\thearg\nothing
\vrule width0pt height\cellsize depth0pt\else
\hbox to 0pt{\usebox2\hss}\fi%
\vbox to 15\unitlength{
\vss
\hbox to 15\unitlength{\hss$#1$\hss}
\vss}}
\newcommand\tableau[1]{\vtop{\let\\=\cr
\setlength\baselineskip{-16000pt}
\setlength\lineskiplimit{16000pt}
\setlength\lineskip{0pt}
\halign{&\cellify{##}\cr#1\crcr}}}
\savebox3{%
\begin{picture}(15,15)
\put(0,0){\line(1,0){15}}
\put(0,0){\line(0,1){15}}
\put(15,0){\line(0,1){15}}
\put(0,15){\line(1,0){15}}
\end{picture}}
\newcommand\expath[1]{%
\hbox to 0pt{\usebox3\hss}%
\vbox to 15\unitlength{
\vss
\hbox to 15\unitlength{\hss$#1$\hss}
\vss}}
\newcommand\bas[1]{\omit \vbox to \cellsize{ \vss \hbox to \cellsize{\hss$#1$\hss} \vss}}

\begin{document}

\title[Classifying the near-equality of ribbon Schur functions]{Classifying the near-equality of ribbon Schur functions}

\author{Foster Tom}
\address{
 Department of Mathematics,
 University of California, Berkeley,
 Berkeley CA 94709, USA}
\email{ftom@berkeley.edu}

\thanks{The author was supported in part by the National Sciences and Engineering Research Council of Canada.}

\begin{abstract} We consider the problem of determining when the difference of two ribbon Schur functions is a single Schur function. We fully classify the five infinite families of pairs of ribbon Schur functions whose difference is a single Schur function with corresponding partition having at most two parts at least $2$. We also prove an identity for differences of ribbon Schur functions and we determine some necessary conditions for such a difference to be Schur-positive, depending on the distribution of $1$'s and the end row lengths.
\end{abstract}

\maketitle
\section{Introduction}\label{sec:intro}

We investigate Schur functions, which form the most esteemed basis for the algebra of symmetric functions. Schur functions arise in representation theory both as the irreducible representations of the symmetric group $S_n$ under the Frobenius map and as the irreducible polynomial representations of the general linear group \cite{youngtab}. In algebraic geometry, their structure constants, the Littlewood--Richardson coefficients, appear as intersection numbers on the Grassmannian \cite{youngtab}. Littlewood--Richardson coefficients also characterize exactly when Hermitian matrices $A$, $B$, and $A+B$ can have prescribed eigenvalues and when modules $\mathcal{C}$, $\mathcal{B}$, and $\mathcal{C}/\mathcal{B}$ over a discrete valuation ring can have prescribed invariant factors \cite{lreigen}. Of particular interest is the problem of determining when a symmetric function is Schur-positive, meaning a nonnegative linear combination of Schur functions, because this is a rare phenomenon \cite{posprob} that suggests a representation-theoretic or geometric interpretation. One way to establish Schur-positivity of a symmetric function is to interpret it as the Frobenius series of a graded $S_n$-module \cite{cdm}. Machinery such as dual equivalence graphs \cite{posdualeq}, crystal bases \cite{crystal}, and Chern plethysm \cite{posboolean} have also been developed to prove Schur-positivity. Schur-positivity has been studied for boolean product polynomials \cite{posboolean}, sets of permutations \cite{possets, possetperm}, labeled binary trees \cite{postrees}, Kazhdan--Lusztig immanants \cite{poskl}, and chromatic symmetric functions of various graphs \cite{chrompostrees, chromposnsym, chrompos31, chromposalg, chromposquasi}. \\

An especially notorious problem is to classify Schur-positivity for differences of skew Schur functions, which are generalizations of Schur functions. Partial results exist \cite{schurpossquare, schurposlogcon}, but it even remains unknown when two skew Schur functions are equal. Fortunately, more is known in the case of ribbon Schur functions, which are a special case of skew Schur functions that are indexed by compositions. Billera, Thomas, and van Willigenburg have classified when two ribbon Schur functions are equal \cite{ribeq}, providing insight towards a combinatorial classification of equality of skew Schur functions \cite{schureqtowards, schureqcoinc}. Necessary and sufficient conditions have been found for the difference of two ribbon Schur functions to be Schur-positive \cite{ribposschubert, ribposnec, ribposmultfree, ribposnecmax} and the sets of nonzero coefficients in the Schur function expansion are fairly well understood \cite{ribeqsupport, ribpossupport}. \\

We study the special case of determining when the difference of two ribbon Schur functions is a single Schur function. After knowing when two ribbon Schur functions are equal, this near-equality phenomenon is the next natural one to investigate and will hopefully provide new insight on Schur-positivity in general. In Section \ref{sec:background} we introduce the necessary background and machinery. In Section \ref{sec:main} we prove our main theorem, which classifies the five infinite families of pairs of ribbon Schur functions whose difference is a single Schur function whose corresponding partition has at most two parts at least $2$. Along the way, we prove an identity for differences of ribbon Schur functions and a necessary condition for Schur-positivity in terms of the end parts of the corresponding compositions. In Section \ref{sec:further} we conclude with our conjectured classification of the sixteen infinite families of near-equality of ribbon Schur functions.

\section{Background}\label{sec:background}
\subsection{Compositions and partitions} \label{subsec:compositions} 

A \emph{composition} is a finite sequence of positive integers $\alpha=\alpha_1\cdots\alpha_R$. The integers $\alpha_i$ are called the \emph{parts} of $\alpha$. The \emph{length} of $\alpha$, denoted $\ell(\alpha)$, is the number of parts $R$ and the \emph{size} of $\alpha$ is the sum of its parts. When $\alpha$ has consecutive equal parts $\alpha_{i+1}=\cdots=\alpha_{i+m}=j$ we will often abbreviate them as $j^m$. By convention, $\alpha_i=0$ for $i>\ell(\alpha)$. The \emph{reverse} of $\alpha$ is the composition $\alpha^*=\alpha_R\cdots\alpha_1$. The \emph{ends} of $\alpha$ is the multiset $$e(\alpha)=\{\alpha_1,\alpha_R\}.$$ A composition $\nu$ is called a \emph{partition} if its parts are weakly decreasing, that is, $\nu_1\geq\nu_2\geq\cdots$. There is a unique partition $\lambda(\alpha)$ \emph{determined} by $\alpha$ given by reordering its parts in weakly decreasing order. Given compositions $\alpha$ and $\beta$, we define the \emph{lexicographic order} by $\alpha>_{lex}\beta$ if $\alpha_i>\beta_i$ at the smallest index $i$ at which they differ. We say that $\beta$ is a \emph{coarsening} of $\alpha$, denoted $\beta\geq_{coar}\alpha$, if $\beta$ can be obtained from summing adjacent parts of $\alpha$.\\

\emph{Throughout this paper, the letters $\alpha$ and $\beta$ will always denote compositions of length $R$ and size $N$ and the letter $\nu$ will always denote a partition of size $N$. We assume that $\alpha\neq 1^R$.\\ }

We also introduce some useful parameters that describe the distribution of the parts of $\alpha$ that are equal to $1$. Let $k=|\{i: \alpha_i=1\}|$ denote the number of such parts and let $\delta_\alpha$ denote the number of such parts in $e(\alpha)$, in other words, $$\delta_\alpha=\chi(\alpha_1=1)+\chi(\alpha_R=1),$$ where for a proposition $P$, $\chi(P)$ is $1$ if $P$ is true and $0$ if $P$ is false. Now writing $\alpha$ as $$\alpha=1^{p_1}z_11^{p_2}z_2\cdots 1^{p_{R-k}}z_{R-k}1^{p_{R-k+1}},$$ where the $p_i\geq 0$ and the $z_i\geq 2$, we define the following sequences of integers. We set $z(\alpha)=z_1\cdots z_{R-k}$. The \emph{profile} of $\alpha$ is $$p(\alpha)=p_1\cdots p_{R-k+1}.$$  We record the number of occurrences of the integer $j$ in the profile of $\alpha$ by defining $$q(\alpha)=q_0q_1\cdots,\hbox{ where }q_j=|\{i: \ p_i=j\}|.$$ It is often useful to subtract the first and last integers of the profile of $\alpha$ by $1$, so we define \begin{align*}p'(\alpha)&=p'_1p'_2\cdots p'_{R-k}p'_{R-k+1}=(p_1-1)p_2\cdots p_{R-k}(p_{R-k+1}-1)\hbox{ and }\\q'(\alpha)&=q'_0q'_1\cdots,\hbox{ where }q'_j=|\{i: \ p'_i=j\}|.\end{align*}

Let us also note that, because only $p'_1$ and $p'_{R-k+1}$ can be negative, \begin{align}\label{eq:sumqj}&\sum_{j\geq 0}q_j=|\{i: \ p_i\geq 0\}|=R-k+1\hbox{ and }\\\nonumber&\sum_{j\geq 0}q'_j=R-k+1-\chi(p'_1=-1)-\chi(p'_{R-k+1}=-1)=R-k-1+\delta_\alpha.\end{align}

\begin{example} \label{ex:comps}
The compositions $\alpha=313$ and $\beta=412$ have length $3$ and size $7$. We have $\alpha^*=313=\alpha$, $e(\alpha)=\{3,3\}$, and $\alpha$ determines the partition $\nu=\lambda(\alpha)=331$. Moreover, $\beta>_{lex}\alpha$ because $\beta_1=4>3=\alpha_1$. Writing $\alpha=313=1^0 \ 3 \ 1^1 \ 3 \ 1^0$, we can read off that $$k=1, \ \delta_\alpha=0, \ z(\alpha)=33, \ p(\alpha)=010, \ p'(\alpha)=(-1)1(-1), \ q(\alpha)=21, \ \hbox{ and }q'(\alpha)=01.$$ 
\end{example}

\subsection{Diagrams and ribbons}\label{subsec:diagrams}

We define the \emph{diagram} of $\nu$ to be the left-justified array of cells with $\nu_i$ cells in the $i$-th row. We use the English convention, where rows are counted from the top. The \emph{conjugate} of $\nu$, denoted $\nu'$, is the partition whose diagram is that of $\nu$ reflected across the diagonal from the top left corner towards the bottom right. Explicitly, $\nu'$ is given by $\nu'_j=|\{i: \ \nu_i\geq j\}|$. We define the \emph{ribbon diagram} of $\alpha$ to be the array of cells with $\alpha_i$ cells in the $i$-th row and where the rightmost cell of the $(i+1)$-th row is directly below the leftmost cell of the $i$-th row. The \emph{transpose} of $\alpha$, denoted $\alpha^t$, is the composition whose ribbon diagram is that of $\alpha$ reflected across the diagonal.

\begin{example} The diagrams of $\nu=331$ and $\nu'=322$ and the ribbon diagrams of $\alpha=313$, $\alpha^t=11311$, $\beta=412$, and $\beta^t=13111$ are shown below. $$\tableau{\ &\ &\ \\ \ &\ &\ \\ \ }\hspace{15pt} \tableau{\ &\ &\ \\ \ &\ \\ \ &\ }\hspace{15pt} \tableau{&& \ &\ &\ \\ && \ \\ \ &\ &\ }\hspace{15pt} \tableau{&& \ \\ && \ \\ \ &\ &\ \\ \ \\ \ }\hspace{15pt} \tableau{& \ &\ &\ &\ \\ & \ \\ \ &\ }\hspace{15pt}\tableau{&& \ \\ \ &\ &\ \\ \ \\ \ \\ \ }$$
\end{example}

Explicitly, $\alpha^t$ is given by $$\alpha^t=(p'_{R-k+1}+2)1^{z_{R-k}-2}(p'_{R-k}+2)\cdots(p'_2+2)1^{z_1-2}(p'_1+2).$$ This is because a row of length $z_i\geq 2$ corresponds to $(z_i-2)$ columns of length one, preceded and followed by columns of length increased by one, and conversely a sequence of $p'_i$ rows of length one corresponds to a column of length $(p'_i+2)$, using the first cell of the previous row and the last cell of the following row. Now the number of parts of $\alpha^t$ equal to $1$ is \begin{align*} (z_{R-k}-2)+\cdots+(z_1-2)&+\chi(p'_1=-1)+\chi(p'_{R-k+1}=-1)\\&=(z_1+\cdots+z_{R-k})-2(R-k)+\chi(\alpha_R\neq 1)+\chi(\alpha_1\neq 1)\\&=(N-k)-2(R-k)+(2-\delta_\alpha)=N-2R+k+2-\delta_\alpha\end{align*} and for $j\geq 2$ the number of parts of $\alpha^t$ equal to $j$ is the number of $i$ such that $p'_i+2=j$, namely $q'_{j-2}$. Thus we have proven that \begin{equation}\label{eq:parttranspose}\lambda(\alpha^t)=\lambda(1^{N-2R+k+2-\delta_\alpha}2^{q'_0}3^{q'_1}\cdots).\end{equation}
In particular, given $\lambda(\alpha)$, one can read off $\lambda(\alpha^t)$ directly from $q'(\alpha)$.
\subsection{Schur functions and ribbon Schur functions}\label{subsec:schrib}
A \emph{semistandard Young tableau (SSYT) of shape $\nu$ (respectively, of ribbon shape $\alpha$)} is a filling $T$ of the cells of the diagram of $\nu$ (respectively, the ribbon diagram of $\alpha$) with positive integers so that the integers in every row are weakly increasing from left to right and the integers in every column are strictly increasing from top to bottom. We use $T_{i,j}$ to refer to the integer in the $i$-th row and $j$-th column of $T$. We also define the \emph{content}  $$\hbox{cont}(T)=\hbox{cont}_1(T)\hbox{cont}_2(T)\cdots$$ 
where $\text{cont}_i(T)$ is the number of $i$'s in $T$. Now for a partition $\nu$ we define the \emph{Schur function} $s_\nu$ to be the formal power series in infinitely many commuting variables $(x_1,x_2,\ldots)$ given by $$s_\nu=\sum_{T\text{ an SSYT of shape }\nu}x_1^{\text{cont}_1(T)}x_2^{\text{cont}_2(T)}\cdots.$$
Similarly, for a composition $\alpha$ we define the \emph{ribbon Schur function} $r_\alpha$ by $$r_\alpha=\sum_{T\text{ an SSYT of ribbon shape }\alpha}x_1^{\text{cont}_1(T)}x_2^{\text{cont}_2(T)}\cdots.$$

\begin{example}\label{ex:ssyt} Below are four SSYTs of shape $\nu=331$ and two SSYTs $T$ and $U$ of ribbon shape $\alpha=313$. We have $T_{1,4}=1$, $\hbox{cont}(T)=421$, $U_{1,4}=2$, and $\hbox{cont}(U)=241$.
$$\tableau{1&1&1\\2&2&4\\6}\hspace{15pt}\tableau{1&1&2\\2&2&4\\6}\hspace{15pt}\tableau{1&2&2\\3&3&4\\6}\hspace{15pt}\tableau{1&2&3\\2&3&4\\6}\hspace{15pt}T=\tableau{& &\ 1&1&1\\ & &\ 2\\1&2&3}\hspace{15pt}U=\tableau{& &\ 1&2&2\\ & &\ 2\\1&2&3}$$
Below are the four terms of $s_{331}$ corresponding to these four SSYTs of shape $331$ and the two terms of $r_{313}$ corresponding to these two SSYTs of ribbon shape $313$. $$s_{331}=\cdots+x_1^3x_2^2x_4x_6+\cdots+x_1^2x_2^3x_4x_6+\cdots+2x_1x_2^2x_3^2x_4x_6+\cdots.$$ $$r_{313}=\cdots+x_1^4x_2^2x_3+\cdots+x_1^2x_2^4x_3+\cdots.$$ \end{example}

The Schur functions $\{s_\nu: \ \nu\hbox{ a partition}\}$ form a basis for the algebra of symmetric functions $\Lambda$ \cite[Corollary 7.10.6]{enum}. Moreover, the ribbon Schur function $r_\alpha$ belongs to $\Lambda$ \cite[Theorem 7.10.2]{enum}. Therefore, we can uniquely expand $r_\alpha$ as a linear combination of Schur functions. In fact, it turns out that $r_\alpha$ is \emph{Schur-positive}, meaning a \emph{nonnegative} linear combination of Schur functions. In general, for symmetric functions $F,G\in\Lambda$, we write $$F\geq_s G$$ if the difference $F-G$ is Schur-positive. Because the Schur functions form a basis for the algebra of symmetric functions, we have $s_\nu=s_\mu$ if and only if $\nu=\mu$. On the other hand, there exist equalities between ribbon Schur functions, in particular, for any composition $\alpha$, we have $r_\alpha=r_{\alpha^*}$ \cite[Theorem 4.1]{ribeq}. Fortunately, for all of the compositions $\alpha$ we will consider, the equality $r_\alpha=r_\beta$ will hold only for $\beta=\alpha$ and $\beta=\alpha^*$. For this reason, we will frequently think of compositions up to reversal when considering their ribbon Schur functions. 

\subsection{Combinatorial tools}\label{subsec:rels}
We now introduce our first main combinatorial tool for calculating ribbon Schur functions. \\

A \emph{Littlewood--Richardson (LR) tableau} of shape $\alpha$ is an SSYT $T$ of ribbon shape $\alpha$ satisfying the \emph{lattice word condition}: in the sequence of entries $a_1\cdots a_N$ of $T$ read from right to left and top to bottom, every initial subsequence $a_1\cdots a_j$ contains at least as many $i$'s as $(i+1)$'s for every $i\geq 1$. We denote by $LR_\alpha$ the set of LR tableaux of shape $\alpha$.

\begin{example}\label{ex:lr} The tableau $T$ from Example \ref{ex:ssyt} is an LR tableau. The tableau $U$ from Example \ref{ex:ssyt} is not an LR tableau because in the sequence $2212321$, the initial substring $2$ contains more $2$'s than $1$'s. Note that the lattice word condition ensures that the content of any LR tableau is a partition. \end{example}

\begin{theorem}\label{thm:lrrule} \cite[Theorem A1.3.3]{enum} (Littlewood--Richardson rule) We have the following identity. $$r_\alpha=\sum_{T\in LR_\alpha}s_{\text{cont}(T)}$$
\end{theorem}

\begin{example}\label{ex:lrrule} The three LR tableaux of shape $\alpha=313$ and the two LR tableaux of shape $\beta=412$ are given below. By the lattice word condition, the rightmost integer must be a $1$ and because the rows must be weakly increasing, the top row must be all $1$'s. In addition, the rightmost column of length at least $2$ must be filled with the integers $1, 2, 3$ in increasing order because the column must be strictly increasing and a $2$ must be read before a $3$.
$$\tableau{& &\ 1&1&1\\ & &\ 2\\1&1&3}\hspace{15pt}\tableau{& &\ 1&1&1\\ & &\ 2\\1&2&3}\hspace{15pt}\tableau{& &\ 1&1&1\\ & &\ 2\\2&2&3}\hspace{15pt}\tableau{& \ 1&1&1&1\\ &\ 2\\1&3}\hspace{15pt}\tableau{& \ 1&1&1&1\\ &\ 2\\2&3}$$
Therefore, by the Littlewood--Richardson rule, we have $$r_{313}=s_{511}+s_{421}+s_{331}\hbox{ and }r_{412}=s_{511}+s_{421}.$$ Note that $$r_{313}-r_{412}=s_{331}.$$ \end{example}

By collecting LR tableaux by content, Theorem \ref{thm:lrrule} can equivalently be stated as $$r_\alpha=\sum_\nu c_{\alpha,\nu}s_\nu,$$ where the \emph{Littlewood--Richardson (LR) coefficient} $c_{\alpha,\nu}$ is the number of LR tableaux of shape $\alpha$ and content $\nu$. Because $c_{\alpha,\nu}\geq 0$, we see that $r_\alpha$ is Schur-positive.\\

Our second main combinatorial tool explores the relationship between Schur functions, ribbon Schur functions, and the basis of complete homogeneous symmetric functions, which we now introduce. For an integer $n$ define the \emph{$n$-th complete homogeneous symmetric function} $$h_n=\sum_{i_1\leq\cdots\leq i_n}x_{i_1}\cdots x_{i_n}.$$ We also set $h_0=1$ and $h_n=0$ when $n<0$. Now for a partition $\nu$ we define the \emph{complete homogeneous symmetric function} $$h_\nu=h_{\nu_1}\cdots h_{\nu_{\ell(\nu)}}.$$

The complete homogeneous symmetric functions $\{h_\nu: \ \nu\hbox{ a partition}\}$ form a basis for the algebra of symmetric functions $\Lambda$ \cite[Corollary 7.6.2]{enum}, hereafter abbreviated as the \emph{$h$-basis}.

\begin{theorem}\cite[Theorem 7.16.1]{enum} \label{thm:jt} (Jacobi--Trudi identity) We have the following identity. $$s_\nu=\det(h_{\nu_i-i+j})_{i,j}$$ \end{theorem}

For a composition $\alpha$ we define $$\mathcal{M}(\alpha)=\{\lambda(\beta): \ \beta\geq_{coar}\alpha\}$$ to be the multiset of partitions determined by coarsenings of $\alpha$. For a multiset $M$ we denote by $m_M(x)$ the multiplicity of $x$ in $M$.

\begin{theorem}\cite[Equation (2.6)]{ribeq} \label{thm:jtrib} We have the following identity. $$r_\alpha=\sum_{\nu\in\mathcal{M}(\alpha)} (-1)^{R-\ell(\nu)}h_\nu$$ \end{theorem}

\begin{example}\label{ex:jt} For $\nu=331$ we expand $s_\nu$ in the $h$-basis as $$s_{331}=\det\left(\begin{matrix}h_3 & h_4 & h_5 \\ h_2 & h_3 & h_4 \\ 0 & 1 & h_1 \end{matrix} \right)=h_{331}-h_{421}-h_{43}+h_{52}.$$
For $\alpha=313$ and $\beta=412$ we have $$\mathcal{M}(\alpha)=\{331, 43, 43, 7\}, \ \mathcal{M}(\beta)=\{421, 43, 52, 7\},\hbox{ and }m_{\mathcal{M}(\alpha)}(43)=2.$$ From these multisets of coarsenings, we see that $$r_{313}=h_{331}-2h_{43}+h_7\hbox{ and }r_{412}=h_{421}-h_{43}-h_{52}+h_7.$$ Once again, we have $$r_{313}-r_{412}=h_{331}-h_{421}-h_{43}+h_{52}=s_{331}.$$\end{example}

Note that $$\lambda(\alpha_1\cdots\alpha_{j-1}(\alpha_j+\alpha_{j+1})\alpha_{j+2}\cdots\alpha_R)>_{lex}\lambda(\alpha_1\cdots\alpha_R)$$ as the partitions will first differ at the smallest $i$ for which $\lambda(\alpha)_i<\alpha_j+\alpha_{j+1}=\lambda(\alpha_1\cdots(\alpha_j+\alpha_{j+1})\cdots\alpha_R)_i$. Therefore, by induction and Theorem \ref{thm:jtrib}, the lexicographically smallest term in the $h$-basis expansion of $r_\alpha$ is $h_{\lambda(\alpha)}$. Moreover, it is a straightforward consequence \cite[Lemma 66]{ribposnecmax} of Theorem \ref{thm:jt} that the lexicographically smallest term of the Schur function expansion of $h_\mu$ is $s_\mu$. Therefore, for some coefficients $b_\mu$ we can write \begin{equation}\label{eq:alslex}r_\alpha=s_{\lambda(\alpha)}+\sum_{\mu>_{lex}\lambda(\alpha)}b_\mu s_\mu.\end{equation} We have the immediate well-known corollary that if $r_\alpha\geq_s r_\beta$, then $\lambda(\alpha)\leq_{lex}\lambda(\beta)$. In fact, this all holds for the dominance order as well, but we will not need this stronger notion.\\

The equation $$r_{313}-r_{412}=s_{331}$$ from Examples \ref{ex:lrrule} and \ref{ex:jt} exhibits the curious situation of two ribbon Schur functions that differ by a single Schur function. This paper aims to classify all those $\alpha$, $\beta$, and $\nu$ for which \begin{equation}\label{eq:ne} r_\alpha-r_\beta=s_\nu.\end{equation}

Such a \emph{near-equality of ribbon Schur functions} would be a cover relation in the Schur-positivity partial order and this relationship is the next most elementary one to investigate after the equality of ribbon Schur functions was classified \cite[Theorem 4.1]{ribeq}. \\

If \eqref{eq:ne} holds, then in particular $r_\alpha\geq_sr_\beta$. It follows \cite[Lemma 3.8]{ribpossupport} that $\alpha$ and $\beta$ must have the same size $N$ and the same length $R$, which is why we make these assumptions, as well as that $\alpha\neq 1^R$. By the Littlewood--Richardson rule, we have \begin{equation} \label{eq:nelr} c_{\alpha,\mu}=\begin{cases} c_{\beta,\mu} & \mu\neq \nu\\ c_{\beta,\mu}+1 & \mu=\nu \end{cases}\end{equation}so one way of studying near-equality is by enumerating LR tableaux. Additionally, if we expand \eqref{eq:ne} in the $h$-basis using Theorem \ref{thm:jt} and Theorem \ref{thm:jtrib}, we see that $s_\nu$ prescribes exactly how the multisets $\mathcal{M}(\alpha)$ and $\mathcal{M}(\beta)$ differ. To be precise, if $s_\nu=\sum_\mu c_\mu h_\mu$, then \begin{equation} \label{eq:nejt} m_{\mathcal{M}(\alpha)}(\mu)-m_{\mathcal{M}(\beta)}(\mu)=(-1)^{R-\ell(\mu)}c_\mu.\end{equation}

Finally, we introduce our third tool from symmetric function theory.

\begin{theorem} \cite[Theorem 7.15.6]{enum}
There is an involutive isomorphism $\omega$ on the algebra $\Lambda$ of symmetric functions, which satisfies $$\omega(r_\alpha)=r_{\alpha^t}\hbox{ and }\omega(s_\nu)=s_{\nu'}.$$
\end{theorem}

In particular, from a near-equality of ribbon Schur functions $$r_\alpha-r_\beta=s_\nu,$$ we can apply the $\omega$ involution to derive a new near-equality $$r_{\alpha^t}-r_{\beta^t}=s_{\nu'}.$$

\begin{example}
By applying the $\omega$ involution to the equation $$r_{313}-r_{412}=s_{331}$$ from Examples \ref{ex:lrrule} and \ref{ex:jt}, we find that $$r_{11311}-r_{13111}=s_{322}.$$
\end{example}

We can also prove the following proposition.\begin{proposition} \label{prop:sedelta} Suppose that $\lambda(\alpha)=\lambda(\beta)$ and $r_\alpha\geq_s r_\beta$. Then $\delta_\alpha\geq\delta_\beta$, and if $\delta_\alpha=\delta_\beta$, then $q'(\alpha)\leq_{lex}q'(\beta)$. \end{proposition}
\begin{proof}
If $r_\alpha\geq_s r_\beta$, then by applying the $\omega$ involution to the Schur-positive difference $r_\alpha-r_\beta$, we find that $r_{\alpha^t}\geq_s r_{\beta^t}$. We have seen from \eqref{eq:alslex} that therefore $\lambda(\alpha^t)\leq_{lex}\lambda(\beta^t)$. Also, because conjugation reverses the lexicographic order on partitions \cite[Section 7.2]{enum}, we must have $\lambda(\alpha^t)'\geq_{lex}\lambda(\beta^t)'$. Finally, from \eqref{eq:sumqj} and \eqref{eq:parttranspose} we have our formula for the conjugate $$\lambda(\alpha^t)'=(N-R+1)(R-k-1+\delta_\alpha)(R-k-1+\delta_\alpha-q'_0(\alpha))(R-k-1+\delta_\alpha-q'_0(\alpha)-q'_1(\alpha))\cdots,$$ from which our conclusion now follows.
\end{proof}

\section{Main theorem}\label{sec:main}

We begin by stating our main theorem, which classifies all near-equalities of ribbon Schur functions for which the partition $\nu$ has at most two parts at least $2$.

\begin{theorem}\label{thm:main}
Suppose that $\nu_3\leq 1$, that is, $\nu=ab1^d$ for some $a$, $b$, $d$. Then $r_\alpha-r_\beta=s_\nu$ if and only if $b\geq 2$ and $\alpha$ and $\beta$ are (up to reversal) as in one of the following five cases.
\begin{align}\label{case:1}
\alpha&=b1^da& &\beta=(b-1)1^d(a+1) \\\label{case:2}
\alpha&=ab1^d& &\beta=(b-1)(a+1)1^d \\\label{case:3}
\alpha&=1^{d+1}a(b-1)& &\beta=1^{d+1}(b-1)a \\\label{case:4}
\alpha&=1a1^d(b-1)& &\beta=1(b-1)1^da\\\label{case:5}
\alpha&=(b-1)1^db(a-b+1)& &\beta=b1^d(b-1)(a-b+1)
\end{align}
\end{theorem}

By applying the $\omega$ involution, we immediately deduce the following corollary.

\begin{corollary} \label{cor:main}
Suppose that $\nu_2\leq 2$, that is, $\nu=a2^c1^d$ for some $a$, $c$, $d$. Then $r_\alpha-r_\beta=s_\nu$ if and only if $c\geq 1$ and $\alpha$ and $\beta$ are (up to reversal) as in one of the following five cases. 
\begin{align}\label{case:6}
\alpha&=1^{c+d}a1^c& &\beta=1^{c+d+1}a1^{c-1} \\\label{case:7}
\alpha&=(a-1)1^{c-1}21^{c+d}& &\beta=(a-1)1^{c+d}21^{c-1} \\\label{case:8}
\alpha&=1^{c-1}21^{c+d-1}a& &\beta=1^{c+d}21^{c-2}a \\\label{case:9}
\alpha&=1^{c-1}a1^{c+d-1}2& &\beta=1^{c+d}a1^{c-2}2 \\\label{case:10}
\alpha&=1^d21^{c-1}a1^{c-1}& &\beta=1^d21^{c-2}a1^c 
\end{align}
\end{corollary}

The remainder of this section is devoted to proving Theorem \ref{thm:main}. Our strategy will be as follows. We first prove in Theorem \ref{thm:forwarddirection} the ``if'' direction, that is, we prove that if $\alpha$ and $\beta$ are as specified, then the near-equality $r_\alpha-r_\beta=s_{ab1^d}$ holds. For the converse, we first show in Proposition \ref{prop:nohooks} that $b\geq 2$. Then we show in Proposition \ref{prop:ab1d} that if $r_\alpha-r_\beta=s_{ab1^d}$, then we must have $e(\alpha)\neq e(\beta)$. Finally, by separately considering the cases where $\lambda(\alpha)\neq\lambda(\beta)$ and where $\lambda(\alpha)=\lambda(\beta)$, we show in Theorem \ref{thm:nelambda} and Theorem \ref{thm:ends} that the five cases of Theorem \ref{thm:main} are the only ones where $e(\alpha)\neq e(\beta)$. 

\subsection{Five families of near-equality}\label{subsec:forwarddirection} In order to prove that near-equality holds in our five families, we first prove a convenient identity for differences of ribbon Schur functions.

\begin{definition}\label{def:specialmovingones}
Let $i$ be minimal with $\alpha_i\geq 2$ and let $i<j\leq R$ and $1\leq t\leq\alpha_i-1$. Now define the composition $$M_{j,t}(\alpha)=\alpha_1\cdots\alpha_{i-1}(\alpha_i-t)\alpha_{i+1}\cdots\alpha_{j-1}(\alpha_j+t)\alpha_{j+1}\cdots\alpha_R,$$ that is, $M_{j,t}(\alpha)$ is formed from $\alpha$ by decrementing the $i$-th part by $t$ and incrementing the $j$-th part by $t$. In addition, define $A_{j,t}(\alpha)$ to be the set of LR tableaux $T$ of shape $\alpha$ such that for some $i\leq j'\leq j-1$ the number of $1$'s in the first $j'$ rows of $T$ does \emph{not} exceed the number of $2$'s in the first $(j'+1)$ rows of $T$ by at least $t$. Finally, define $B_{j,t}(\alpha)$ to be the set of LR tableaux $U$ of shape $M_{j,t}(\alpha)$ such that $$\hbox{if }U_{j,j_1}=\cdots=U_{j,j_1+t-1}=1,\hbox{ then }j\leq R-1\hbox{ and }U_{j,j_1+t}\geq U_{j+1,j_1},$$ where the leftmost cell of row $j$ in $U$ is in column $j_1$. In other words, if the $j$-th row of $U$ begins with $t$ $1$'s, then the $(t+1)$-th entry of this row must be greater than or equal to the rightmost entry of the row below.
\end{definition}

We are now ready to state our ribbon difference identity. We will work through an example before supplying the proof.

\begin{theorem}\label{thm:specialmovingones}

Let $i$ be minimal with $\alpha_i\geq 2$ and let $i<j\leq R$ and $1\leq t\leq\alpha_i-1$. Then we have the following identity. $$r_\alpha-r_{M_{j,t}(\alpha)}=\sum_{T\in A_{j,t}(\alpha)}s_{\text{cont}(T)}-\sum_{U\in B_{j,t}(\alpha)}s_{\text{cont}(U)}$$\end{theorem}

\begin{example} \label{ex:specialmovingones} Let $\alpha=1116311$, so that $i=4$, and let $j=5$ and $t=3$, so that $\beta=M_{5,3}(\alpha)=1113611$.\\

By the lattice word condition, the first four rows of any $T\in LR_\alpha$ must be filled as follows.

$$\tableau{&&&&&&&1\\&&&&&&&2\\&&&&&&&3\\&&1&1&1&1&1&4\\\ &\ &\ \\\ \\\ }$$

Now $A_{5,3}(\alpha)$ is the set of such $T$ for which the number of $1$'s in the first four rows does not exceed the number of $2$'s in the first five rows by at least three. As there are presently six $1$'s in the first four rows of $T$ and one $2$ in the first five rows, the fifth row must be filled with all $2$'s. The LR tableaux of $A_{5,3}(\alpha)$ are enumerated below.

$$\tableau{&&&&&&&1\\&&&&&&&2\\&&&&&&&3\\&&1&1&1&1&1&4\\2&2&2\\3\\4}\hspace{15pt} \tableau{&&&&&&&1\\&&&&&&&2\\&&&&&&&3\\&&1&1&1&1&1&4\\2&2&2\\3\\5} \hspace{15pt} \tableau{&&&&&&&1\\&&&&&&&2\\&&&&&&&3\\&&1&1&1&1&1&4\\2&2&2\\5\\6}$$

By the lattice word condition, the first four rows of any $U\in LR_\beta$ must be filled as follows.

$$\tableau{&&&&&&&1\\&&&&&&&2\\&&&&&&&3\\&&&&&1&1&4\\ \ & \ & \ &\ &\ &\ &\\\ \\\ }$$

Now $B_{5,3}(\alpha)$ is the set of such $U$ for which if $U_{5,1}=U_{5,2}=U_{5,3}=1$, then $U_{5,4}\geq U_{6,1}$. Because the fifth row of $U$ can have at most three numbers at least $2$, namely two $2$'s and one $5$, then we indeed have $U_{5,1}=U_{5,2}=U_{5,3}=1$ and $U_{5,4}=2$ and so $U_{6,1}=2$. The LR tableaux of $B_{5,3}(\alpha)$ are enumerated below.

$$\tableau{&&&&&&&1\\&&&&&&&2\\&&&&&&&3\\&&&&&1&1&4\\1&1&1&2&2&5&\\2\\3}\hspace{15pt} \tableau{&&&&&&&1\\&&&&&&&2\\&&&&&&&3\\&&&&&1&1&4\\1&1&1&2&2&5&\\2\\6} $$

Finally, by Theorem \ref{thm:specialmovingones}, the difference $r_\alpha-r_\beta$ is \begin{align*}r_{1116311}-r_{1113611}&=\sum_{T\in A_{5,3}(\alpha)}s_{\text{cont}(T)}-\sum_{U\in B_{5,3}(\alpha)}s_{\text{cont}(U)}\\&=(s_{6422}+s_{64211}+s_{641111})-(s_{64211}+s_{641111})=s_{6422}.\end{align*}

\end{example}

We will now prove Theorem \ref{thm:specialmovingones}.\\

\emph{Proof of Theorem \ref{thm:specialmovingones}.} For ease of notation, set $\beta=M_{j,t}(\alpha)$, $A=A_{j,t}(\alpha)$, and $B=B_{j,t}(\alpha)$. We will construct a content-preserving bijection $$f:(LR_\alpha\setminus A)\rightarrow (LR_\beta\setminus B),$$ from which it immediately follows that $$ r_\alpha-r_\beta=\sum_{T\in LR_\alpha} s_{\text{cont}(T)}-\sum_{U\in LR_\beta} s_{\text{cont}(U)}\\\nonumber=\sum_{T\in A}s_{\text{cont}(T)}-\sum_{U\in B}s_{\text{cont}(U)}.$$

Given an LR tableau $T\in LR_\alpha\setminus A$, we construct $f(T)$ as the tableau of shape $\beta$ where the $i$-th row is filled with $(\beta_i-1)$ $1$'s followed by an $i$, the $j$-th row is filled with $t$ $1$'s followed by the entries in the $j$-th row of $T$, and all other rows are filled as in $T$. Informally, we remove $t$ $1$'s from the $i$-th row of $T$ and append them to the front of the $j$-th row of $T$ to create $f(T)$.\\

As an illustration, if $\alpha=1116311$, $j=5$, and $t=3$ as in Example \ref{ex:specialmovingones}, and $T$ is the tableau to the left, then $f(T)$ is the tableau to the right. Informally, we move the $3$ red $1$'s.

$$\tableau{&&&&&&&1\\&&&&&&&2\\&&&&&&&3\\&&\textcolor{red}{1}&\textcolor{red}{1}&\textcolor{red}{1}&1&1&4\\1&2&5\\3\\6}\hspace{15pt} \tableau{&&&&&&&1\\&&&&&&&2\\&&&&&&&3\\&&&&&1&1&4\\\textcolor{red}{1}&\textcolor{red}{1}&\textcolor{red}{1}&1&2&5&\\3\\6} $$

We first check that $f(T)\in LR_\beta\setminus B$. By construction, $f(T)$ is of shape $\beta$ and has the same content as $T$. The $j$-th row of $f(T)$ is still weakly increasing because the $1$'s were added to the front. In the case that $i\geq 2$ and $t=\alpha_i-1$, we need to check that the rightmost column of $f(T)$ is still strictly increasing. However, because $T\notin A$, the number of $1$'s in the first $i$ rows of $T$, namely $\alpha_i$, must exceed the number of $2$'s in the first $(i+1)$ rows of $T$ by at least $t=\alpha_i-1$. So there is at most one $2$ in the first $(i+1)$ rows of $T$, which is in the second row, and so the rightmost entry of the $(i+1)$-th row is not a $2$ and must be an $(i+1)$. To check the lattice word condition, note that since the reading word of $f(T)$ differs from that of $T$, which is a lattice word, only by moving $t$ $1$'s from the $i$-th to the $j$-th row, it suffices to check that there are not too many $2$'s in this range. Now again since $T\notin A$, the number of $1$'s in $T$ exceeds the number of $2$'s by at least $t$ from the $i$-th row to the $j$-th row, and so indeed $f(T)\in LR_\beta$. Finally, $f(T)\notin B$ because the first $t$ entries of the $j$-th row of $f(T)$ are all $1$'s, and if $j<R$, then the $(t+1)$-th entry of the $j$-th row, which in $T$ was directly above the rightmost entry of the $(j+1)$-th row, cannot be greater than or equal to it.\\

Conversely, given an LR tableau $U\in LR_\beta\setminus B$, then we construct $f^{-1}(U)$ as the tableau of shape $\alpha$ where the $i$-th row is filled with $(\alpha_i-1)$ $1$'s followed by an $i$, the $j$-th row is filled with the rightmost $(\beta_j-t)$ entries of $U$, and all other rows are filled as in $U$. Because the first $t$ entries of the $j$-th row of $U$ are all $1$'s and the $(t+1)$-th entry of this row is strictly smaller than the rightmost entry of the row below, and because $t$ $1$'s were moved to the $i$-th row we have $f^{-1}(U)\in LR_\alpha\setminus A$. By construction, $f^{-1}(f(T))=T$ and $f(f^{-1}(U))=U$, so $f$ is a bijection.\qed\\ 

Now that we have Theorem \ref{thm:specialmovingones} at our disposal, we are able to calculate differences of ribbon Schur functions more efficiently.

\begin{theorem}\label{thm:forwarddirection} In each of the five cases of Theorem \ref{thm:main}, we have $r_\alpha-r_\beta=s_{ab1^d}$. \end{theorem}

\begin{proof} Because the proofs are similar, we only prove Case \ref{case:1}, that is, when $\bm{\alpha=b1^da}$ and $\bm{\beta=(b-1)1^d(a+1)}$. Note that $\beta=M_{d+2,1}(\alpha)$ so we will apply Theorem \ref{thm:specialmovingones}.\\

Any tableau $T\in A_{d+2,1}(\alpha)$ has all $1$'s in its top row, the integers $1$ through $(d+2)$ in its column of length $(d+2)$, and only $1$'s and $2$'s in its remaining cells. The total number of $2$'s must be at most $b$ because there are $b$ $1$'s placed so far, and must be at least $b$ by definition of $A_{d+2,1}(\alpha)$, so there is a unique $T\in A_{d+2,1}(\alpha)$, which has content $\hbox{cont}(T)=ab1^d$. In the concrete case where $a=6$, $b=4$, and $d=2$, this tableau $T$ is given below.

$$T=\tableau{&&&&&& 1 & 1 & 1 & 1&\\ &&&&&& 2 \\ &&&&&& 3 \\ & 1 &1 &2 &2 &2 &4 }$$

Any tableau $U\in B_{d+2,1}(\alpha)$ has all $1$'s in its top row, the integers $1$ through $(d+2)$ in its column of length $(d+2)$, and only $1$'s and $2$'s in its remaining cells. The total number of $2$'s must be at most $(b-1)$ because there are $(b-1)$ $1$'s placed so far, and therefore there must be a $1$ in the leftmost cell because $(b-1)+1<a+1$. However, by definition of $B_{d+2,1}(\alpha)$, there cannot be a $1$ in the leftmost cell because $d+2=R$, so in fact $B_{d+2,1}(\alpha)$ is empty. Therefore, by Theorem \ref{thm:specialmovingones} we have $$r_{b1^da}-r_{(b-1)1^d(a+1)}=\sum_{T\in A_{d+2,1}(\alpha)}s_{\text{cont}(T)}-\sum_{U\in B_{d+2,1}(\alpha)}s_{\text{cont}(U)}=s_{ab1^d}.$$
\end{proof}

Now our goal for the remainder of this paper is to prove the converse that if $r_\alpha-r_\beta=s_{ab1^d}$ then $\alpha$ and $\beta$ must be (up to reversal) as in one of the five cases of Theorem \ref{thm:main}.

\subsection{Necessary conditions for near-equality} We begin by showing that if $r_\alpha-r_\beta=s_{ab1^d}$, then we must have $b\geq 2$. 

\begin{proposition} \label{prop:nohooks} We do not have $r_\alpha-r_\beta=s_{a1^d}$ for any $\alpha$ and $\beta$. 
\end{proposition}

\begin{proof}
Suppose that $r_\alpha-r_\beta=s_{a1^d}$. By Theorem \ref{thm:jt}, the $h$-basis expansion of $s_{a1^d}$ has exactly one term, namely $h_{a1^d}$, whose partition has exactly $(d+1)$ parts. Therefore, if $\mu$ is a partition with $\ell(\mu)=d+1$, then by \eqref{eq:nejt} we must have \begin{equation}\label{eq:partsmults}m_{\mathcal{M}(\alpha)}(\mu)=\begin{cases}m_{\mathcal{M}(\beta)}(\mu)+(-1)^{R-d-1} & \mu=a1^d\\m_{\mathcal{M}(\beta)}(\mu) & \mu\neq a1^d.\end{cases}\end{equation}
However, $\alpha$ and $\beta$ have the same total number of coarsenings of length $(d+1)$, specifically $\binom{R-1}{R-(d+1)}=\binom{R-1}d$, because from the $(R-1)$ pairs of adjacent parts we must choose to add $(R-(d+1))$ times to produce a coarsening of length $(d+1)$. This contradicts \eqref{eq:partsmults}.
\end{proof}

Now we show that if $r_\alpha-r_\beta=s_{ab1^d}$, then $\alpha$ and $\beta$ must have different ends.

\begin{proposition} \label{prop:ab1d} If $r_\alpha-r_\beta=s_{ab1^d}$, then $e(\alpha)\neq e(\beta)$. \end{proposition}

\begin{proof}
Suppose that $r_\alpha-r_\beta=s_{ab1^d}$ and $e(\alpha)=e(\beta)$, and note that by Proposition \ref{prop:nohooks}, we must have $b\geq 2$. By Theorem \ref{thm:jt}, the $h$-basis expansion of $s_{ab1^d}$ contains exactly one term, namely $(-1)^{d+1}h_{(a+d+1)(b-1)}$, that has an $(a+d+1)$ part. Therefore if $\mu$ is a partition of $N$ with an $(a+d+1)$ part, meaning in particular that $\mu_1=a+d+1$ because $$a+d+1>a+d=\frac{2a+2d}2\geq\frac{a+b+d}2=\frac N2,$$ then by \eqref{eq:nejt} we must have \begin{equation}\label{eq:ab1dmults} m_{\mathcal{M}(\alpha)}(\mu)=\begin{cases} m_{\mathcal{M}(\beta)}(\mu)+(-1)^{R+d+1} & \mu=(a+d+1)(b-1)\\ m_{\mathcal{M}(\beta)}(\mu) & \mu\neq (a+d+1)(b-1).\end{cases}\end{equation}

A coarsening $\gamma\geq_{coar}\alpha$ with an $(a+d+1)$ part arises from summing a string of consecutive parts $\alpha_{i+1}\cdots\alpha_{R-j}$ in $\alpha$ with sum $(a+d+1)$, and then from further summing some of the $i$ parts to the left and some of the $j$ parts to the right. There are $2^{\max\{i-1,0\}}$ ways to sum some of the $i$ parts to the left and $2^{\max\{j-1,0\}}$ ways to sum some of the $j$ parts to the right because for each pair of adjacent parts, we can choose to add or not. Also note that because $a+d+1>\frac N2$, $i$ and $j$ can be determined from $\gamma$ so there will be no double-counting. Therefore, setting $$S_\alpha=\{(i,j): \ \alpha_{i+1}+\cdots+\alpha_{R-j}=a+d+1\},$$ we have the total number of coarsenings of $\alpha$ with an $(a+d+1)$ part is $$\sum_{\mu: \ \mu_1=a+d+1} m_{\mathcal{M}(\alpha)}(\mu)=\sum_{(i,j)\in S_\alpha}2^{\max\{i-1,0\}+\max\{j-1,0\}}.$$ Separating those $(i,j)$ with $i,j\leq 1$ if they appear, we have \begin{align}\nonumber \sum_{\mu: \ \mu_1=a+d+1} m_{\mathcal{M}(\alpha)}(\mu)&=\sum_{\substack{(i,j)\in S_\alpha:\\ i\geq 2\text{ or }j\geq 2}}2^{\max\{i-1,0\}+\max\{j-1,0\}}\\\nonumber&+\chi((0,0)\in S_\alpha)+\chi((1,0)\in S_\alpha)+\chi((0,1)\in S_\alpha)+\chi((1,1)\in S_\alpha)\end{align} and similarly for $\beta$. Now we cannot have $(0,0)\in S_\alpha$ because $N=a+b+d>a+d+1$ since $b\geq 2$. Meanwhile, $(1,0)\in S_\alpha$ if and only if $\alpha_1=b-1$, $(0,1)\in S_\alpha$ if and only if $\alpha_R=b-1$, and $(1,1)\in S_\alpha$ if and only if $\alpha_1+\alpha_R=b-1$. So because $e(\alpha)=e(\beta)$, the sum of these terms is identical for $\alpha$ and $\beta$, and so \begin{align*}\sum_{\mu: \ \mu_1=a+d+1} m_{\mathcal{M}(\alpha)}(\mu) \ -&\sum_{\mu: \ \mu_1=a+d+1} m_{\mathcal{M}(\beta)}(\mu)\\=&\sum_{\substack{(i,j)\in S_\alpha:\\i\geq 2\text{ or }j\geq 2}}2^{\max\{i-1,0\}+\max\{j-1,0\}}-\sum_{\substack{(i,j)\in S_\beta:\\i\geq 2\text{ or }j\geq 2}}2^{\max\{i-1,0\}+\max\{j-1,0\}},\end{align*} which is even because all of the terms are even. This contradicts \eqref{eq:ab1dmults}.
\end{proof}

\subsection{Near-equality with different ends} Given Proposition \ref{prop:ab1d}, we can now focus our attention on classifying near-equality when $\alpha$ and $\beta$ have different ends. We first address the case where $\alpha$ and $\beta$ have different parts. In fact, we prove something a little stronger.

\begin{theorem} \label{thm:nelambda} Suppose that $r_\alpha-r_\beta=s_\nu$ and $\lambda(\alpha)\neq\lambda(\beta)$. Then $\nu=ab1^d$ for some $a,b,d$ and $\alpha$ and $\beta$ are (up to reversal) as in Case \eqref{case:1} or \eqref{case:2} of Theorem \ref{thm:main}, namely
\begin{itemize}
\item[Case \eqref{case:1}: ]$\alpha=b1^da$, $\beta=(b-1)1^d(a+1)$
\item[Case \eqref{case:2}: ] $\alpha=ab1^d$, $\beta=(b-1)(a+1)1^d$. \end{itemize} \end{theorem}

\begin{proof}
From \eqref{eq:alslex}, we have $\nu=\lambda(\alpha)<_{dom}\lambda(\beta)$, and by Proposition \ref{prop:nohooks}, we have $\nu_2\geq 2$. Now if $\nu_3\geq 2$, then the $h$-basis expansion of $s_\nu$ has at least four terms with $R$ parts, namely $$h_\nu, \ -h_{\lambda(\nu_1(\nu_2+1)(\nu_3-1)\cdots\nu_R)}, \ -h_{\lambda((\nu_1+1)(\nu_2-1)\nu_3\cdots\nu_R)},\hbox{ and }h_{\lambda((\nu_1+2)(\nu_2-1)(\nu_3-1)\cdots\nu_R)},$$ which is impossible because $r_\alpha-r_\beta$ has two such terms, namely $h_{\lambda(\alpha)}$ and $-h_{\lambda(\beta)}$. So $\nu=\lambda(\alpha)=ab1^d$ for some $a,b,d$, and because the $h$-basis expansion of $s_{ab1^d}$ contains the term $-h_{(a+1)(b-1)1^d}$ with $R$ parts, we have $\lambda(\beta)=(a+1)(b-1)1^d$. If $d=0$ then we are done, so suppose that $d\geq 1$.\\

The partition $\mu=ab21^{d-2}<_{lex}(a+1)(b-1)1^d=\lambda(\beta)$ so does not appear in $\mathcal{M}(\beta)$, so because by Theorem \ref{thm:jt} the $h$-basis expansion of $s_\nu$ contains the term $-(d-1)h_\mu$, we must have by \eqref{eq:nejt} that $m_{\mathcal{M}(\alpha)}(\mu)=d-1$. This partition arises from a coarsening of $\alpha$ precisely when an adjacent pair among the $d$ $1$'s of $\alpha$ is summed. In order for there to be $(d-1)$ such adjacent pairs, all of the $1$'s must be together and so $\alpha$ is up to reversal one of $b1^da$, $ab1^d$, or $ba1^d$. If $a=b$, then $\alpha=b1^da$ or $\alpha=ab1^d$ up to reversal. It $a\neq b$, then the partition $\tau=a(b+1)1^{d-1}<_{lex}\lambda(\beta)$ so does not appear in $\mathcal{M}(\beta)$, so because by Theorem \ref{thm:jt} the $h$-basis expansion of $s_\nu$ contains the term $-h_\tau$, we must have by \eqref{eq:nejt} that $m_{\mathcal{M}(\alpha)}(\tau)=1$, so the $b$ is next to a $1$, meaning again that $\alpha=b1^da$ or $\alpha=ab1^d$ up to reversal.\\

Finally, if $\alpha=b1^da$, then $r_\beta=r_{b1^da}-s_{ab1^d}=r_{(b-1)1^d(a+1)}$ by Theorem \ref{thm:forwarddirection}. Therefore, by Proposition \ref{prop:sedelta} the number of end $1$'s must be the same, that is, $\delta_\beta=\chi(b=2)$, which means that up to reversal $\beta=(b-1)1^d(a+1)$. If $\alpha=ab1^d$ then $r_\beta=r_{ab1^d}-s_{ab1^d}=r_{(b-1)(a+1)1^d}$ by Theorem \ref{thm:forwarddirection}. Therefore, by Proposition \ref{prop:sedelta} the number of end $1$'s must be the same and also $q'(\alpha)=q'(\beta)$, which means that up to reversal either $\beta=(b-1)(a+1)1^d$ or $\beta=(a+1)(b-1)1^d$ with $b\geq 3$. In order to exclude this last possibility, observe that $(a+1)b1^{d-1}\in\mathcal{M}((a+1)(b-1)1^d)\setminus\mathcal{M}((b-1)(a+1)1^d)$, and therefore $r_{(a+1)(b-1)1^d}\neq r_{(b-1)(a+1)1^{d-1}}$, so we can only have up to reversal $\beta=(b-1)(a+1)1^d$.
\end{proof}

By applying the $\omega$ involution, we immediately deduce the following corollary.

\begin{corollary} \label{cor:nelambda} Suppose that $r_\alpha-r_\beta=s_\nu$ and $\lambda(\alpha^t)\neq\lambda(\beta^t)$. Then $\nu=a2^c1^d$ and $\alpha$ and $\beta$ are (up to reversal) as in Case \eqref{case:6} or \eqref{case:7} of Corollary \ref{cor:main}, namely
\begin{itemize}
\item[Case \eqref{case:6}: ]$\alpha=1^{c+d}a1^c$, $\beta=1^{c+d+1}a1^{c-1}$
\item[Case \eqref{case:7}: ] $\alpha=(a-1)1^{c-1}21^{c+d}$, $\beta=(a-1)1^{c+d}21^{c-1}$. \end{itemize} \end{corollary}

Now the last step in proving Theorem \ref{thm:main} will be to prove the following theorem, which classifies all cases of near-equality for which $\lambda(\alpha)=\lambda(\beta)$, $\lambda(\alpha^t)=\lambda(\beta^t)$, and $e(\alpha)\neq e(\beta)$.

\begin{theorem} \label{thm:ends}
Suppose that $r_\alpha-r_\beta=s_\nu$, $\lambda(\alpha)=\lambda(\beta)$, $\lambda(\alpha^t)=\lambda(\beta^t)$, and $e(\alpha)\neq e(\beta)$. Then $\nu=ab1^d$ for some $a,b,d$ with $b\geq 3$ and $\alpha$ and $\beta$ are (up to reversal) as in Case \eqref{case:3}, \eqref{case:4}, or \eqref{case:5} of Theorem \ref{thm:main}, namely \begin{itemize}
\item [Case \eqref{case:3}: ]$\alpha=1^{d+1}a(b-1)$, $\beta=1^{d+1}(b-1)a$
\item [Case \eqref{case:4}: ]$\alpha=1a1^d(b-1)$, $\beta=1(b-1)1^da$
\item [Case \eqref{case:5}: ]$\alpha=(b-1)1^db(a-b+1)$, $\beta=b1^d(b-1)(a-b+1)$.\end{itemize} \end{theorem}

Our first task towards proving Theorem \ref{thm:ends} is to count certain LR coefficients in Lemma \ref{lem:countends}, which we will show in Lemma \ref{lem:compareends} are sensitive to the ends of a composition. The following definition will be cryptic but we will work through an example.

\begin{definition} \label{def:modifiedends} Recall that $z(\alpha)=z_1\cdots z_{R-k}$ are the parts of $\alpha$ not equal to $1$. We define the sequence of nonnegative integers \begin{align*}\epsilon(\alpha)&=\epsilon_0(\alpha)\cdots\epsilon_{R-k-1}(\alpha)\\&=(z_1-2+\chi(p_1=0))(z_2-2)\cdots(z_{R-k-1}-2)(z_{R-k}-2+\chi(p_{R-k+1}=0)).\end{align*}
Also let $S'=\sum_{j\geq 1}q'_j$. Then for $0\leq M\leq\frac12(N-2R+k+2-\delta_\alpha)$ and $1\leq u\leq S'$, unless $S'=0$, in which case $u=0$, we define the partition of size $N$ $$\mu(M,u)=(N-R+1-M)(R-k-1+\delta_\alpha+M)u1^{k-\delta_\alpha-u}.$$
\end{definition}

\begin{example} \label{ex:ends}
Let $\alpha=31311515$. Then $N=20$, $R=8$, $k=4$, $\delta_\alpha=0$, $e(\alpha)=\{3,5\}$, $$\epsilon(\alpha)=2134, \ p'(\alpha)=(-1)121(-1), \ q'(\alpha)=021, \ \hbox{ and }S'=2+1=3.$$ Now for $0\leq M\leq \frac12(20-16+4+2-0)=5$ and $1\leq u\leq S'=3$ we define the partition $$\mu(M,u)=(13-M)(M+3)u1^{4-u}.$$
\end{example}

We now count the number of LR tableaux of content $\mu(M,u)$.

\begin{lemma} \label{lem:deep} In the ribbon diagram of $\alpha$, there are exactly $(k-\delta_\alpha)$ cells $x$ with the property that there are at least two cells above $x$ in the same column as $x$. These $(k-\delta_\alpha)$ cells occur in exactly $S'$ columns. \end{lemma}

\begin{proof} Such cells arise precisely in columns of length at least $3$, of which there are exactly $S'=\sum_{j\geq 1}q'_j$ by \eqref{eq:parttranspose}. A column of length $(p'_i+2)$ gives rise to exactly $\max\{p'_i,0\}$ such cells. So the number of such cells is $$\sum_{i=1}^{R-k+1}p'_i+\chi(p'_1=-1)+\chi(p'_{R-k+1}=-1)=(k-2)+(2-\delta_\alpha)=k-\delta_\alpha.$$ \end{proof}

\begin{lemma} \label{lem:countends} If $0\leq M\leq\epsilon_0(\alpha)$, then the number of LR tableaux of shape $\alpha$ and content $\mu(M,u)$ is $$c_{\alpha,\mu(M,u)}=\binom{S'-1}{u-1}|E_{\alpha,M}|,$$ where $E_{\alpha,M}\subset\mathbb{Z}^{R-k-1}$ is the set of lattice points $$E_{\alpha,M}=\{x_1\cdots x_{R-k-1}:\ \sum_{i=1}^{R-k-1}x_i=M,\ 0\leq x_i\leq \epsilon_i(\alpha)\hbox{ for }1\leq i\leq R-k-1\}.$$\end{lemma}

We present how the proof works in an example before diving into the details.

\begin{example} \label{ex:countends} Let $\alpha=31311515$. By Example \ref{ex:ends}, we have $k-\delta_\alpha=4$, $S'=3$, $\epsilon(\alpha)=2134$, and we are considering LR tableaux of shape $\alpha$ and content $$\mu(M,u)=(13-M)(M+3)u1^{4-u}$$ for $0\leq M\leq \epsilon_0(\alpha)=2$ and $1\leq u\leq 3$.\\

By Lemma \ref{lem:deep}, there are exactly $k-\delta_\alpha=4$ cells, occupying $S'=3$ columns, which have at least two cells above in the same column and which therefore must be filled by the $u+(4-u)=4$ available entries at least $3$. Of the $u$ $3$'s, one must occupy the top such cell and no two may appear in the same column, so there are $\binom{S'-1}{u-1}=\binom 2{u-1}$ choices of in which columns to place the remaining $(u-1)$ $3$'s, after which the remaining $(4-u)$ cells must be filled with the $(4-u)$ entries at least $4$ in increasing order.\\

We must then place the $1$'s and $2$'s. The top two cells of the $R-k-1+\delta_\alpha=3$ columns of length at least $2$ must be filled with a $1$ above a $2$ and the first row must be filled with $1$'s. The two fillings with $u=2$ are shown below.

$$\tableau{&&&&&&&&&&1 &1 &1 \\ &&&&&&&&&&2 \\ &&&&&&&&1 &\ &3 \\ &&&&&&&&2 \\ &&&&&&&&3 \\ &&&&1 &\ &\ &\ &4 \\&&&&2 \\\ &\ &\ &\ &5}\tableau{&&&&&&&&&&1 &1 &1 \\ &&&&&&&&&&2 \\ &&&&&&&&1 &\ &3 \\ &&&&&&&&2 \\ &&&&&&&&4 \\ &&&&1 &\ &\ &\ &5 \\&&&&2 \\\ &\ &\ &\ &3}$$

We now have $M$ $2$'s left to place. There are presently $\epsilon_0(\alpha)=2$ more $1$'s than $2$'s placed and so because $M\leq 2$ we do not need to worry about the lattice word condition. Because the rows must be weakly increasing, we need only specify how many $2$'s will occupy each of our remaining three rows of length at least $2$. The $i$-th such row from the top has $\epsilon_i(\alpha)$ vacant cells and so can be filled with $x_i$ additional $2$'s, where $0\leq x_i\leq \epsilon_i(\alpha)$, so the number of ways to place the remaining $M$ $2$'s is exactly \begin{equation}\label{eq:endscalc} |\{x_1x_2x_3: \ \sum_{i=1}^3x_i=M, \ 0\leq x_1\leq 1, \ 0\leq x_2\leq 3, \ 0\leq x_3\leq 4\}|=|E_{\alpha,M}|.\end{equation} For an example where $M=2$, we have $|E_{\alpha,2}|=|\{110,101,020,011,002\}|=5$.
\end{example}

Now the proof of Lemma \ref{lem:countends} works exactly as in this example.\\

\emph{Proof of Lemma \ref{lem:countends}. } Consider an LR tableaux $T$ of shape $\alpha$ and content $\mu(M,u)$. By Lemma \ref{lem:deep}, there are exactly $(k-\delta_\alpha)$ cells, occupying $S'$ columns, which have at least two cells above in the same column and which therefore must be filled by the $u+(k-\delta_\alpha-u)$ available entries at least $3$. Of the $u$ $3$'s, one must occupy the top such cell and no two may appear in the same column, so there are $\binom{S'-1}{u-1}$ choices of in which columns to place the remaining $(u-1)$ $3$'s, after which the remaining $(k-\delta_\alpha-u)$ cells must be filled with the $(k-\delta_\alpha-u)$ entries at least $4$ in increasing order.\\

We must then place the $1$'s and $2$'s. By \eqref{eq:sumqj} and \eqref{eq:parttranspose}, there are $$\sum_{j\geq 0}q'_j=R-k-1+\delta_\alpha$$ columns of length at least $2$. The top two cells of these columns must be filled with a $1$ above a $2$ and the rest of the top row of length at least two must be filled with $1$'s. We now have $M$ $2$'s left to place. There are presently $(z_1-2)$ more $1$'s than $2$'s placed, unless $p_1=0$, in which case there are $(z_1-1)$, and so because $$M\leq\epsilon_0(\alpha)=z_1-2+\chi(p_1=0),$$ we do not need to worry about the lattice word condition. Because the rows must be weakly increasing, we need only specify how many $2$'s will occupy each of the remaining $(R-k-1)$ rows of length at least $2$. The $i$-th such row from the top has its first and last cells occupied so has $(z_i-2)$ vacant cells, unless $i=R-k$ and $p_{R-k+1}=0$, in which case only its last cell is occupied and so has $(z_i-1)$ vacant cells. Therefore, the number $x_i$ of additional $2$'s that can be placed in this row satisfies $$0\leq x_i\leq \epsilon_i(\alpha),$$ so the number of ways to place the remaining $M$ $2$'s is exactly $|E_{\alpha,M}|$. Putting this together, we have the number of LR tableaux of shape $\alpha$ and content $\mu(M,u)$ is $c_{\alpha,\mu(M,u)}=\binom{S'-1}{u-1}|E_{\alpha,M}|$, as desired.\qed\\

Now we will start looking at compositions in pairs in order to show that indeed the LR coefficient above informs us about the ends of a composition.

\begin{lemma} \label{lem:compareends}
Suppose that $\lambda(\alpha)=\lambda(\beta)$ and $\delta_\alpha=\delta_\beta$. Assume that $\alpha_1\leq \alpha_R$ and $\beta_1\leq\beta_R$. \begin{enumerate}
\item \label{part:compare1}If $\alpha_1<\beta_1$, then $|E_{\alpha,\epsilon_0(\alpha)}|=|E_{\beta,\epsilon_0(\alpha)}|+1+\chi(\alpha_1=\alpha_R)$.
\item \label{part:compare2}If $\alpha_1=\beta_1$ and $\alpha_R<\beta_R$, then $|E_{\alpha^*,\epsilon_0(\alpha^*)}|=|E_{\beta^*,\epsilon_0(\alpha^*)}|+1$.
\item \label{part:comparesame}If $\alpha_1=\beta_1$ and $\alpha_R=\beta_R$, then $|E_{\alpha,M}|=|E_{\beta,M}|$ for every $M\leq \frac12(N-2R+k+2-\delta_\alpha)$. 
\end{enumerate}
\end{lemma}

Note that because $r_\alpha=r_{\alpha^*}$, the assumptions $\alpha_1\leq\alpha_R$ and $\beta_1\leq\beta_R$ are of little concern. We illustrate with an example before examining the technicalities.

\begin{example} Let $\alpha=31311515$ as in Example \ref{ex:countends} and let $\beta=51113315$. Then $\epsilon(\beta)=4114$ and taking $M=\epsilon_0(\alpha)=2$, we have $$E_{\beta,\epsilon_0(\alpha)}=\{x_1x_2x_3: \ \sum_{i=1}^3 x_i=2, \ 0\leq x_1\leq 1, \ 0\leq x_2\leq 1, \ 0\leq x_3\leq 4\}=\{110,101,011,002\}.$$ Comparing this with \eqref{eq:endscalc}, we see that the only difference is the constraint on $x_2$. There we had the constraint $x_2\leq 3$, which was superfluous because $x_1+x_2+x_3=M=2$, while here we have the constraint $x_2\leq 1$, which specifically excludes the single point $020$ and as a result, $|E_{\alpha,2}|=|E_{\beta,2}|+1$. 
\end{example}

\emph{Proof of Lemma \ref{lem:compareends}. } Because the proofs are very similar, we only prove the first part. Note that because $\delta_\alpha=\delta_\beta$, we must have $\alpha_1\neq 1$ and so $\delta_\alpha=\delta_\beta=0$, $\epsilon_0(\alpha)=\alpha_1-1<\beta_1-1=\epsilon_0(\beta)$, $\epsilon_{R-k-1}(\alpha)=\alpha_R-1$, and $\epsilon_{R-k-1}(\beta)=\beta_R-1$.\\

Because $\lambda(\alpha)=\lambda(\beta)$, the parts $z(\alpha)$ and $z(\beta)$ of $\alpha$ and $\beta$ are only permuted and so there are $i,j,i',j'$ such that $z_1(\alpha)=z_i(\beta)$, $z_{R-k}(\alpha)=z_j(\beta)$, $z_1(\beta)=z_{i'}(\alpha)$, and $z_{R-k}(\beta)=z_{j'}(\alpha)$. Now excluding these parts, we have $$\{\epsilon_1(\alpha),\ldots,\epsilon_{R-k-2}(\alpha)\}\setminus\{\epsilon_{i'}(\alpha),\epsilon_{j'}(\alpha)\}=\{\epsilon_1(\beta),\ldots,\epsilon_{R-k-2}(\beta)\}\setminus\{\epsilon_i(\beta),\epsilon_j(\beta)\}$$ as multisets so we can re-enumerate these as $\{\epsilon_1,\ldots,\epsilon_{R-k-4}\}$. By permuting the constraints, which does not change the sizes of the sets, we now have \begin{align}\nonumber |E_{\alpha,\epsilon_0(\alpha)}|=|\{&x_1\cdots x_{R-k-1}: \ \sum_{i=1}^{R-k-1}x_i=\alpha_1-1, \ 0\leq x_i\leq\epsilon_i\hbox{ for }1\leq i\leq R-k-4,\\\nonumber& 0\leq x_{R-k-3}\leq \beta_1-2, \ 0\leq x_{R-k-2}\leq \beta_R-2, \ 0\leq x_{R-k-1}\leq \alpha_R-1 \}|\\\nonumber |E_{\beta,\epsilon_0(\alpha)}|=|\{&x_1\cdots x_{R-k-1}: \ \sum_{i=1}^{R-k-1}x_i=\alpha_1-1, \ 0\leq x_i\leq\epsilon_i\hbox{ for }1\leq i\leq R-k-4,\\\nonumber& 0\leq x_{R-k-3}\leq \alpha_1-2, \ 0\leq x_{R-k-2}\leq \alpha_R-2, \ 0\leq x_{R-k-1}\leq \beta_R-1 \}|. \end{align} We see that the constraints $0\leq x_i\leq\epsilon_i$ for $1\leq i\leq R-k-4$ are identical for the two sets. Because $\alpha_1-1\leq\beta_1-2\leq\beta_R-2$ and $\alpha_1-1\leq\alpha_R-1$ by hypothesis, the last three constraints for $E_{\alpha,\epsilon_0(\alpha)}$ are superfluous and similarly, because $\alpha_1-1\leq\beta_R-1$, the last constraint for $E_{\beta,\epsilon_0(\alpha)}$ is too. However, the constraint $x_{R-k-3}\leq\alpha_1-2$ for $E_{\beta,\epsilon_0(\alpha)}$ specifically excludes the single point $0\cdots0(\alpha_1-1)00$ and if $\alpha_1=\alpha_R$, the constraint $x_{R-k-2}\leq\alpha_R-2$ specifically excludes the single point $0\cdots0(\alpha_1-1)0$. Each of these appear in the first set. Therefore, we have $|E_{\alpha,\epsilon_0(\alpha)}|=|E_{\beta,\epsilon_0(\alpha)}|+1+\chi(\alpha_1=\alpha_R)$.\qed \\

Now Lemma \ref{lem:countends} and Lemma \ref{lem:compareends} allow us to identify specific partitions at which the LR coefficients for $\alpha$ and $\beta$ differ. One immediate consequence is the following necessary condition for Schur-positivity of a difference $r_\alpha-r_\beta$, which generalizes \cite[Theorem 40]{ribposnecmax}.

\begin{theorem} \label{thm:generalshortends} Suppose that $\lambda(\alpha)=\lambda(\beta)$. Assume that $\alpha_1\leq\alpha_R$ and $\beta_1\leq\beta_R$. Then the compositions $\alpha_1\alpha_R$ and $\beta_1\beta_R$ satisfy $$\hbox{ if }r_\alpha\geq_sr_\beta\hbox{, then }\alpha_1\alpha_R\leq_{lex}\beta_1\beta_R.$$ \end{theorem}

\begin{proof}
Before we can apply Lemma \ref{lem:compareends}, we must first address the case where $\delta_\alpha\neq\delta_\beta$. By Proposition \ref{prop:sedelta}, we would have $\delta_\alpha>\delta_\beta$, from which it follows that $\alpha_1\alpha_R\leq_{lex}\beta_1\beta_R$, as desired. Now we may assume that $\delta_\alpha=\delta_\beta$.\\

If $\beta_1<\alpha_1$, then by Lemma \ref{lem:countends} and Lemma \ref{lem:compareends}, Part \ref{part:compare1}, we would have  $$c_{\beta,\mu(\epsilon_0(\alpha),S')}=|E_{\beta,\epsilon_0(\alpha)}|>|E_{\alpha,\epsilon_0(\alpha)}|=c_{\alpha,\mu(\epsilon_0(\alpha),S')},$$ contradicting $r_\alpha\geq_s r_\beta$ by Theorem \ref{thm:lrrule}, and so $\alpha_1\leq \beta_1$. Similarly, if $\alpha_1=\beta_1$ and $\beta_R>\alpha_R$, then by Lemma \ref{lem:countends} and Lemma \ref{lem:compareends}, Part \ref{part:compare2}, we would have  $$c_{\beta^*,\mu(\epsilon_0(\alpha^*),S')}=|E_{\beta^*,\epsilon_0(\alpha^*)}|>|E_{\alpha^*,\epsilon_0(\alpha^*)}|=c_{\alpha^*,\mu(\epsilon_0(\alpha^*),S')},$$ contradicting $r_\alpha\geq_s r_\beta$ by Theorem \ref{thm:lrrule} because $r_\alpha=r_{\alpha^*}$, and so $\alpha_R\leq\beta_R$, as desired.
\end{proof}

Our next task will be to investigate the following statistic. 

\begin{definition} \label{def:ap} The \emph{adjacent pairs} of $\alpha$ is the multiset of multisets $$ap(\alpha)=\{\{\alpha_1,\alpha_2\},\{\alpha_2,\alpha_3\},\ldots,\{\alpha_{R-1},\alpha_R\}\}.$$\end{definition}

Because coarsenings of $\alpha$ with length $(R-1)$ are of the form $\alpha_1\cdots\alpha_{i-1}(\alpha_i+\alpha_{i+1})\alpha_{i+2}\cdots\alpha_R$, we see that $ap(\alpha)$ precisely encodes these coarsenings, the longest aside from $\alpha$ itself.

\begin{proposition} \label{prop:eap} Suppose that $\lambda(\alpha)=\lambda(\beta)$. If $e(\alpha)\neq e(\beta)$, then $ap(\alpha)\neq ap(\beta)$. \end{proposition}

\begin{proof}
Every part $\alpha_i$ belongs to two adjacent pairs, namely $\{\alpha_{i-1},\alpha_i\}$ and $\{\alpha_i,\alpha_{i+1}\}$, with the exception of $\alpha_1$ and $\alpha_R$, which belong to only one. Therefore, we can read off $e(\alpha)$ from how many times each integer of $\lambda(\alpha)$ is present in $ap(\alpha)$.
\end{proof}

\begin{lemma} \label{lem:ap} Suppose that $r_\alpha-r_\beta=s_{ab1^d}$, $\lambda(\alpha)=\lambda(\beta)$, and $ap(\alpha)\neq ap(\beta)$. Then $$\lambda(\alpha)=a(b-1)1^{d+1} \hbox{ or }\lambda(\alpha)=\lambda((a-b+1)b(b-1)1^d).$$ \end{lemma}

\begin{proof}
Recall that $b\geq 2$ by Proposition \ref{prop:nohooks}. Because $ap(\alpha)\neq ap(\beta)$, the multisets $\mathcal{M}(\alpha)$ and $\mathcal{M}(\beta)$ differ at a partition with $(R-1)$ parts determined by a coarsening that arises from joining an adjacent pair of parts in $\alpha$ or $\beta$. In particular, because $ab1^d$ is a partition of greatest length at which $\mathcal{M}(\alpha)$ and $\mathcal{M}(\beta)$ differ, we must have $\ell(ab1^d)=R-1$ and that this partition is of the form \begin{equation}\label{eq:apnu} ab1^d=\lambda((\alpha_i+\alpha_j)\alpha_1\cdots\alpha_{i-1}\alpha_{i+1}\cdots\alpha_{j-1}\alpha_{j+1}\cdots\alpha_R)\end{equation} for some $i,j$. Moreover, by Theorem \ref{thm:jt} the $h$-basis expansion of $s_{ab1^d}$ contains the term $-h_{(a+1)(b-1)1^d}$ so $\mathcal{M}(\alpha)$ and $\mathcal{M}(\beta)$ also differ at this partition, which has length $(R-1)$ and is therefore also of the form above. By \eqref{eq:alslex}, we must have $\lambda(\alpha)\leq_{lex}ab1^d$, so $(a+1)$ cannot be a part of $\alpha$, and so we have \begin{equation}\label{eq:aplamb} \lambda(\alpha)=xy(b-1)1^d\end{equation} for some positive integers $x,y$ with $x+y=a+1$. Comparing with \eqref{eq:apnu}, the $(b-1)$ must be summed with another part of $\lambda(\alpha)$ to make either the $b$, in which case $$\{x,y\}=\{a,1\}\hbox{ and }\lambda(\alpha)=a(b-1)1^{d+1},$$ or the $a$, in which case $$\{x,y\}=\{a-b+1,b\}\hbox{ and }\lambda(\alpha)=\lambda((a-b+1)b(b-1)1^d).$$
\end{proof}

We are now abundantly prepared to prove Theorem \ref{thm:ends}.\\

\emph{Proof of Theorem \ref{thm:ends}. } By reversing $\alpha$ and $\beta$ if necessary, we may assume without loss of generality that $\alpha_1\leq\alpha_R$ and $\beta_1\leq\beta_R$. Also because $\lambda(\alpha)=\lambda(\beta)$ and $\lambda(\alpha^t)=\lambda(\beta^t)$, we have by \eqref{eq:parttranspose} that $q'(\alpha)=q'(\beta)$ and $\delta_\alpha=\delta_\beta$. By Theorem \ref{thm:generalshortends}, we must have  $\alpha_1\alpha_R\leq_{lex}\beta_1\beta_R$, and more specifically because $\delta_\alpha=\delta_\beta$ and $e(\alpha)\neq e(\beta)$, we have either $$2\leq\alpha_1<\beta_1\hbox{, or }\alpha_1=\beta_1\hbox{ and } 2\leq\alpha_R<\beta_R.$$

Therefore, taking $$M=\epsilon_0(\alpha)=\alpha_1-1<\beta_1-1,\hbox{ or }M=\epsilon_0(\alpha^*)=\alpha_R-1<\beta_R-1$$ respectively, we have by Lemma \ref{lem:countends} and Lemma \ref{lem:compareends} that the LR coefficients $$c_{\alpha,\mu(M,u)}\neq c_{\beta,\mu(M,u)}$$ for $1\leq u\leq S'$, unless $S'=0$, in which case $u=0$. If $S'\geq 2$, then $c_{\alpha,\mu(M,1)}\neq c_{\beta,\mu(M,1)}$ and $c_{\alpha,\mu(M,2)}\neq c_{\beta,\mu(M,2)}$, violating \eqref{eq:nelr}, and so we must have $S'=\sum_{j\geq 1} q'_j\leq 1$, meaning that all of the $1$'s in $\alpha$ and $\beta$ are together, with possibly the exception of lone $1$'s on the end. Now we have $c_{\alpha,\mu(M,S')}\neq c_{\beta,\mu(M,S')}$, and so by \eqref{eq:nelr} the partition $\nu$ is  $$\nu=\mu(M,S')=(N-R+1-M)(R-k-1+\delta_\alpha+M)1^{k-\delta_\alpha},$$ which is a partition of the form $\nu=ab1^d$ where $a=N-R+1-M$, $b=R-k-1+\delta_\alpha+M$, and $d=k-\delta_\alpha$. \\

Again because $e(\alpha)\neq e(\beta)$, we have by Proposition \ref{prop:eap} that $ap(\alpha)\neq ap(\beta)$, so by Lemma \ref{lem:ap} and since $\nu=ab1^d$, we have $\lambda(\alpha)=a(b-1)1^{d+1}$ or $\lambda(\alpha)=\lambda((a-b+1)b(b-1)1^d)$. Note that because $R=d+3$, we have $$b=R-k-1+\delta_\alpha+M=(d+3)-d-1+M\hbox{, so }b-1=M+1,$$ and because either $M=\alpha_1-1<\beta_1-1$ or $M=\alpha_R-1<\beta_R-1$, we see that in either case $b-1\in e(\alpha)\setminus e(\beta)$. Also, because $\delta_\alpha=\delta_\beta$, we must have $b\geq 3$.\\

If $\lambda(\alpha)=a(b-1)1^{d+1}$, the only possibilities of $\alpha$ and $\beta$ (up to reversal) satisfying our conditions $S'\leq 1$, $q'(\alpha)=q'(\beta)$, $\delta_\alpha=\delta_\beta$, and $b-1\in e(\alpha)\setminus e(\beta)$ are \begin{itemize}
\item $\alpha=1^{d+1}a(b-1)$, $\beta=1^{d+1}(b-1)a$
\item $\alpha=1a1^d(b-1)$, $\beta=1(b-1)1^da$. \end{itemize}
These are Cases \eqref{case:3} and \eqref{case:4}, as desired. \\

On the other hand, if $\lambda(\alpha)=\lambda((a-b+1)b(b-1)1^d)$, then because the number of $1$'s of $\alpha$ is $k=d$ and $d=k-\delta_\alpha$, we have $\delta_\alpha=\delta_\beta=0$. Additionally, because $h_{ab1^d}$ and $h_{(a+1)(b-1)1^d}$ are the only terms in the $h$-basis expansion of $s_\nu$ with $(R-1)$ parts, and because these coarsenings arise precisely from joining adjacent pairs in $\alpha$ or $\beta$, then by \eqref{eq:nejt} we have \begin{align}\label{eq:endsapconds} m_{ap(\alpha)}(\{b,a-b+1\})&=m_{ap(\beta)}(\{b,a-b+1\})+1\\\nonumber m_{ap(\beta)}(\{b-1,a-b+1\})&=m_{ap(\alpha)}(\{b-1,a-b+1\})+1\\\nonumber m_{ap(\alpha)}(\{x,y\})&=m_{ap(\beta)}(\{x,y\})\hbox{ otherwise.}\end{align}

Now the only possibility of $\alpha$ and $\beta$ (up to reversal) satisfying our conditions $S'\leq 1$, $q'(\alpha)=q'(\beta)$, $\delta_\alpha=\delta_\beta=0$, $b-1\in e(\alpha)\setminus e(\beta)$, and \eqref{eq:endsapconds} is \begin{itemize}
\item $\alpha=(b-1)1^db(a-b+1)$, $\beta=b1^d(b-1)(a-b+1)$.
\end{itemize} This is Case \eqref{case:5}, as desired.\qed\\

Theorem \ref{thm:ends}, along with Theorem \ref{thm:forwarddirection}, Proposition \ref{prop:nohooks}, Proposition \ref{prop:ab1d}, and Theorem \ref{thm:nelambda} completes the proof of Theorem \ref{thm:main}. 

\section{Further directions} \label{sec:further}
We conclude with our conjecture about the classification of near-equality of ribbon Schur functions in general.

\begin{conjecture}\label{conj:classify} Suppose that $r_\alpha-r_\beta=s_\nu$. Then $\alpha$, $\beta$, and $\nu$ are (up to reversal of $\alpha$ and $\beta$) as in one of the five cases of Theorem \ref{thm:main}, one of the five cases of Corollary \ref{cor:main}, or one of the following six cases.
\begin{align*}
\alpha&=1^{c+d+1}a(b-1)1^c& &\beta=1^{c+d+1}(b-1)a1^c& &\nu=ab2^c1^d& \\
\alpha&=(b-1)1^{c-1}21^{c+d}a& &\beta=(b-1)1^{c+d}21^{c-1}a& &\nu=ab2^c1^d& \\
\alpha&=1^ca(b-1)1^{c-1}21^d& &\beta=1^c(b-1)a1^{c-1}21^d& &\nu=ab2^c1^d& \\
\alpha&=(a-b+1)(b-1)1^{c-1}21^{c+d}(b-1)& &\beta=(a-b+1)(b-1)1^{c+d}21^{c-1}(b-1)& &\nu=ab2^c1^d\\
\alpha&=2a121& &\beta=212a1& &\nu=a42&\\
\alpha&=231^{d+2}21& &\beta=21^{d+2}231& &\nu=33221^d&
\end{align*} In particular, we conjecture that $\nu$ must be of the form $\nu=ab2^c1^d$, in other words, $\nu_3\leq 2$. \end{conjecture}

It is not too difficult to prove the forward direction, that is, that near-equality holds in these six cases, by using Theorem \ref{thm:specialmovingones} and applying the $\omega$ involution. This conjecture has been verified by computer for $N\leq 16$. \\

We hope that the techniques presented may provide some insight towards proving Conjecture \ref{conj:classify} and in the study of calculating with symmetric functions in general.

\section*{Acknowledgements} 

The author would like to thank Stephanie van Willigenburg for suggesting this problem, for many hours of insightful discussion, and for her thorough feedback on this paper. \\

The author would like to thank Andrew Rechnitzer for his thoughtful comments.

\end{document}